\title{Homological stability properties of spaces of rational J-holomorphic curves in $\p$}
\author{
        Jeremy Miller \\
                Department of Mathematics\\
        Stanford University\\
        450 Serra Mall, Stanford, CA
}
\date{\today}

\documentclass[12pt]{article}

\usepackage {amsmath, amscd, amsbsy, amsfonts}
\usepackage{latexsym,amssymb,mathrsfs,bbm}
\usepackage[all]{xy}
\usepackage{graphicx}

\newcounter{prob}
\newcommand{\Z}{\mathbb{Z}}
\newcommand{\C}{\mathbb{C}}
\newcommand{\M}{\mathcal{M}}
\newcommand{\p}{\mathbb{P}^2}

\newcommand{\s}{\mathbb{P}^1}
\newcommand{\R}{\mathbb{R}}
\newcommand{\m}{\longrightarrow}
\newcommand{\w}{\omega}
\newcommand{\n}{\bar d_\nu}
\newcommand{\ft}{Fr$\acute{e}$chet }

\newtheorem{theorem}{Theorem}[section]
\newtheorem{lemma}[theorem]{Lemma}

\newtheorem{corollary}[theorem]{Corollary}

\newtheorem{definition}{Definition}[section]

\newenvironment{proof}[1][Proof]{\begin{trivlist}
\item[\hskip \labelsep {\bfseries #1}]}{\end{trivlist}}

\newcommand{\qed}{\nobreak \ifvmode \relax \else
      \ifdim\lastskip<1.5em \hskip-\lastskip
      \hskip1.5em plus0em minus0.5em \fi \nobreak
      \vrule height0.75em width0.5em depth0.25em\fi}

\begin{document}
\maketitle

\begin{abstract}
In a well known work [Se], Graeme Segal proved that the space of holomorphic maps from a  Riemann surface to a complex projective space is homology equivalent to the corresponding continuous mapping space through a range of dimensions increasing with degree.  In this paper, we address if a similar result holds when other (not necessarily integrable) almost complex structures are put on projective space. We take almost complex structures that are compatible with the underlying symplectic structure. We obtain the following result: the inclusion of the space of based degree $k$ $J$-holomorphic maps from $\s$ to $\p$ into the double loop space of $\p$ is a homology surjection for dimensions $j\leq 3k-3$. The proof involves constructing a gluing map analytically in a way similar to McDuff and Salamon in [MS] and Sikorav in [S] and then comparing it to a combinatorial gluing map studied by Cohen, Cohen,  Mann, and Milgram in [CCMM].

\end{abstract}

\section{Introduction}

The space of holomorphic maps between two complex manifolds can be topologized as a subspace of the space of all continuous maps. A natural question is, what can be said about the relative homotopy type of the space of continuous maps relative to the subspace of holomorphic maps. Let us fix some notation.

\begin{definition}

For a complex manifold $M$, $A\in H_2(M)$, let $Hol_A(\s,M)$ denote the space of holomorphic maps $u:\s \m M$ with $u_* [\s ]=A$. Fix a point $m_0 \in M$. Let $Hol^*_A(\s,M)$ denote the subspace of  maps $u$ with $u(\infty)=m_0$. 

\end{definition}

\begin{definition}

Let $\Omega^2 M$ denote the space of based continuous maps from $\s$ to $M$. For $A\in H_2(M)$, let $\Omega^2_A M$ denote the subspace of $\Omega^2 M$ consisting of maps $u$ such that $u_* [\s ]=A$.

\end{definition}

Using the orientation coming from the complex structure, there is a natural isomorphism $H_2(\mathbb{P}^n) = \Z$. The integer corresponding to a homology class is called degree. In [Se], G. Segal prove the following theorem.

\begin{theorem}
The inclusion map $i:Hol_k^*(\s,\mathbb{P}^n) \m \Omega^2_k \mathbb{P}^n$ induces an isomorphism on homotopy groups $\pi_i$ for $i<k(2n-1)$ and a surjection for $i=k(n-1)$.
\end{theorem}

There have been many generalizations of this theorem including replacing $\mathbb{P}^n$ by some other projective variety. We address a different question. In a conversation with Ralph Cohen, Yakov Eliashberg posed the following question: Is Segal's theorem still true if the standard complex structure on $\mathbb{P}^n$ is replaced by some other almost complex structure $J$ on $\mathbb{P}^n$ compatible with the standard symplectic form?

The study of $J$-holomorphic curves for general almost complex structures $J$ has been important in symplectic geometry since Gromov's foundational paper [G]. Allowing non-integrable almost complex structures is important for many transverality arguments involving moduli spaces of $J$-holomorphic curves. These curves are also important because embedded $J$-holomorphic curves are symplectic surfaces and hence can be used to solve symplectic isotopy problems [S]. The topology of the space of all $J$-holomorphic curves can depend on the choice of almost complex structure. For example, in [A], Abreu considers the case of a symplectic form $\omega$ on $S^2 \times S^2$ where the two spheres have different areas. He shows that the space of $J$-holomorphic curves representing the homology class of the smaller sphere can be empty or non-empty depending on the choice of $\omega$-compatible almost complex structure. The purpose of our paper is to show that some aspects of the topology of the space of $J$-holomorphic curves in $\p$ is independent of the choice of almost complex structure. Specifically, we prove the following theorem.

\begin{theorem} 

If $J$ is an almost complex structure compatible with the standard symplectic form on $\p$, then $i:Hol_{k}^*(\s,(\p,J)) \m \Omega^2_k \p$ is a homology surjection for all dimensions $\leq 3k-3$.

\end{theorem}

This result is a generalization of Segal's theorem in [Se] in the sense that general almost complex structures are considered, but it is a weakening of Segal's theorem in the sense that homotopy equivalence is replaced by homology surjection. Segal's original proof is based on analyzing various configuration spaces of points labeled with integers. It can be shown that all holomorphic maps are algebraic by an elementary Liouville's theorem argument. Monic polynomials are determined by their roots. This allows one to study holomorphic mapping spaces by studing configuration spaces. However, for general almost complex structure $J$, $J$-holomorphic functions are no longer $n+1$-tuples of polynomials.

When the target is $\s$, the question of generalizing Segal's theorem is not interesting. All almost complex structures on real 2 dimensional manifolds are integrable and $\s$ admits a unique complex structure up to diffeomorphism. Thus Segal's original theorem applies to all almost complex structures on $\s$. While the question of generalizing Segal's theorem is interesting for all $\mathbb{P}^n$ with $n\geq 2$, the problem seems to be more approachable if $n=2$. This is because of the phenomenon of automatic transversality in dimension 4 discovered by Gromov in [G].

Using Gromov compactness and automatic tranversality, Gromov proved that the topology of the space of degree one rational $J$-holomorphic maps to $\p$ is independent of choice almost complex structure. We review this in section 2. We leverage this result about degree one curves to draw conclusions about higher degree mapping spaces using a gluing argument. If $J_0$ is the standard integrable complex structure, $Hol^*(\s,(\p,J_0))$ has a little 2-disks operad action given by juxtaposition of roots [BM]. This gives a gluing map $C_2(k) \times_{\Sigma_k} Hol^*_1(\s,(\p,J_0))^k \m Hol^*_k(\s,(\p,J_0))$ where $C_2(k)$ is the k'th space of the little 2-disks operad. In [CCMM], it is shown that this map is a surjection on homology. These results are reviewed in section 3. In section 4, we construct a similar map for arbitrary almost complex structures. In [MS], it is shown that the existence of a gluing map is equivalent to the surjectivity of a certain linearized $\bar d$ operator. We use automatic transversality to prove surjectivity of this linearized $\bar d$ operator using the ideas of Sikorav in [S]. In section 5, we show that the two gluing maps are homotopic which allows us to deduce Theorem 1.2.

The results of this paper are part of my doctoral thesis, written at Stanford University under the supervision of R. Cohen. I would
like to thank him as well as E. Ionel for their advice and support.

\section{Degree one maps and automatic transversality}

This section is devoted to reviewing Gromov's proof that diffeomorphism type of $Hol_1(\s,(\p,J))/PSL_2(\C)$ is independent of $J$. We will also prove that the evaluation map $Hol_1(\s,(\p,J)) \m \p$ is a smooth fiber bundle with fiber $Hol^*_1(\s,(\p,J))$. Before we can describe the proofs of these theorems, we need to review basic definitions and facts about moduli spaces of $J$-holomorphic curves and maps. A more complete discussion can be found in [MS]. We also review the theory of automatic transversality developed in [G], [HLS], and [S]. In this paper, $M$ will be a compact symplectic manifold with symplectic form $\omega$.

\begin{definition}

An almost complex structure $J$ is a section of $Hom_\R(TM,TM)$ such that $J^2=-id$.

\end{definition}

For a symplectic form $\omega$ and almost complex structure $J$, we can construct an associated bilinear form $g\in T^*M \otimes T^*M$ as follows. If $v$ and $w$ are tangent vectors based at the same point, let $g(v,w)=\omega(v,Jw)$. 

\begin{definition}

An almost complex structure $J$ is said to be compatible with $\omega$ if the associated bilinear form is a Riemannian  metric.

\end{definition}

Throughout the paper, all almost complex structures are required to be compatible with the symplectic form and of class $C^{\infty}$. The sphere $\s$ will always be considered with the almost complex structure induced from the standard complex structure. We will denote this almost complex structure by $j$.

\begin{definition}
Consider $C^{\infty}(\s,M)$ as an infinite dimensional Fr$\acute{e}$chet manifold topologized with the $C^{\infty}$ topology. Let $\Upsilon \m C^{\infty}(\s,M)$ be the infinite dimensional Fr$\acute{e}$chet bundle whose fiber over a map $u$ is $\Omega^{0,1}(u^*TM)$. Here $\Omega^{0,1}(u^*TM)$ denotes the space of smooth anti-complex linear one forms with values in the pullback of the tangent bundle of $M$. 

\end{definition}

\begin{definition}

Let $u: \s \m M$ be smooth. The non-linear $\bar d$ operator $\bar d^{nl}$ is defined by the formula $\bar d^{nl}(u)=\frac{1}{2}(Du+J \circ Du \circ j) \in \Omega^{0,1}(u^*TM)$. The map $u$ is said to be $J$-holomorphic if $\bar d^{nl}(u)=0$. The operator $\bar d^{nl}$ can be viewed as a section of $\Upsilon$.

\end{definition}

\begin{definition}
For $A\in H_2(M)$, let $Hol_A(\s,(M,J)$ denote the subspace of $C^{\infty}(\s,M)$ of maps $u$ with $\bar d^{nl} u=0$ and $u_* [\s ]=A$. Fix $m_0 \in M$. Let $Hol_A^*(\s,(M,J)$ denote the subspace of $Hol_A(\s,(M,J))$ consisting of maps $u$ with $u(\infty)=m_0$.

\end{definition}

The noncompact group $PSL_2(\C)=Hol_1(\s,\s)$ acts on $Hol_A(\s,(M,J))$ via precomposition. This action reparameterizes the map but keeps the image fixed. The action is not always free. The stabilizers of this action are finite provided $A \neq 0$. See [MS] for a discussion of this topic. 

\begin{definition}

A non-constant $J$-holomorphic map $u:\s \m M$ is called a multiple cover if $u=w \circ v$  with $v: \s \m \s$ a holomorphic branched cover of degree greater than $1$ and $w: \s \m M$ a $J$-holomorphic map.

\end{definition}

\begin{theorem}
Let $A\in H_2(M)$ be a non-zero homology class. If $u \in Hol_A(\s,(M,J))$ is fixed by $g\in PSL_2(\C)$ with $g \neq id$, then $u$ is a multiple cover. 
\end{theorem}

See [MS] for a proof. Note that the action $PSL_2(\C)$ restricts to an action on the subset of non-multiply covered maps. 

\begin{definition}

Let $\M_A(\s,(M,J))=\{u \in Hol_A(\s, M)$ such that u is not a multiple cover$\}/PSL_2(\C)$. For a finite subset  $Y \subset M$, let $\M_A(\s,(M,J),Y)$ denote the subspace of $\M_A(\s,(M,J))$ consisting of curves passing though $Y$.

\end{definition}

Elements in the above moduli space will be called unparameterized rational curves or rational curves for short. In [MS], it is shown that for generic choice of $J$, the above space is a smooth manifold. However, since we will only be interested in four dimensional manifolds, we can instead use more powerful automatic transversality arguments. The proof that these spaces are smooth manifolds involves identifying their tangent spaces. This involves defining linearized $\bar d$ operators.

The tangent bundle to the \ft  manifold $C^\infty({\s,M})$ is a bundle whose fiber at a map $u$ is $\Gamma(u^*TM)$. Here the symbol $\Gamma(V)$ denotes the vector space of smooth sections of a vector bundle $V$. Recall the definition of the bundle $\Upsilon$ from above. The operator $\bar d^{nl}$ defines a section of $\Upsilon$. Thus the derivative $D \bar d^{nl}$ is a bundle map $D \bar d ^{nl}:TC^{\infty}(\s,M) \m T\Upsilon$. Over a point in the zero section of $\Upsilon$, $(u,0)$, $T\Upsilon_{(u,0)}$ splits into a direct sum of $\Upsilon_u$ and $T_u C^{\infty}(\s,M)$. In [MS], McDuff and Salamon explain how to use a Hermitian connection to construct a continuous family of projections $\pi_u: T\Upsilon_{(u,0)} \m \Upsilon_u$. They note that this projection is independent of choice of Hermitian connection if $u$ is $J$-holomorphic. This construction allows us to define a linearized $\bar d$ operator.

\begin{definition}
For $u:\s \m M$ smooth, let $\bar d_u: \Gamma(u^*TM) \m \Omega^{0,1}(u^*TM)$ be given by the formula $\bar d_u = \pi_u \circ D \bar d ^{nl}_u$.

\end{definition}

From now on, assume $\dim_\R M =4$. Fix a $J$-holomorphic immersion, $u: \s \m M$. The pullback bundle $u^*TM$ is a complex plane bundle over $\s$. Since $u$ is an immersion, $u^*TM/T\s$ is a line bundle, which we denote $N$ for normal bundle. 

\begin{definition}
Let $\n : \Gamma(N) \m \Omega^{0,1}(N)$ be defined as follows. Lift a section $\xi^n \in \Gamma(N)$ to a section $\xi \in \Gamma(u^*TM)$. Let $\n \xi^n$ be the image of $\bar d_u \xi$ in $\Omega^{0,1}(N)$.

\end{definition}

See [S] for a proof that $\n$ is well defined. 

\begin{definition}
Let $\bar d_\tau: \Gamma(T\s) \m \Omega^{0,1}(T\s)$ be defined to be the linearized $\bar d$ operator associated to the identity function on $\s$ considered with respect to the standard almost complex structure in the domain and range.
\end{definition}

By section 4.1 of [S], there is a commuting diagram of exact sequences:

$$
\begin{array}{lcccccccl}
0  &\m &\Gamma(T\s) &\m  & \Gamma(u^* TM) & \m & \Gamma(N) & \m & 0                \\
\downarrow  & & \bar{d}_\tau \downarrow  &    & \bar{d}_u \downarrow &  & \n \downarrow &  & \downarrow              \\

0 & \m  &\Omega^{0,1}(T\s) &\m & \Omega^{0,1}(u^*TM) & \m &\Omega^{0,1}(N) &\m & 0
\end{array}$$

In order to prove that a moduli space or holomorphic mapping space is a manifold, one first identifies the relevant linearized $\bar d$ operator.  If that operator is surjective, one then uses an implicit function theorem argument to identify a neighborhood of the point in the moduli space or mapping space with a neigborhood of $0$ in the kernel of the $\bar d$ operator. In this way, we can view the tangent space to the moduli space or mapping space as the kernel of the $\bar d$ operator. When surjective, the three linearized $\bar d$ operators from above correspond to the following moduli spaces and mapping spaces. The kernel of $\bar{d}_\tau$ can be identified with the tangent space to the space of reparameterizations, $PSL_2(\C)$, at the identity. The kernel of $\bar{d}_u$, can be identified with the tangent space to $Hol_A(\s,(M,J))$ at $u$, and the kernel of $\n$ can be identified with the tangent space of $\M_A(\s,(M,J))$ at the equivalence class of $u$. It can be shown that $\bar{d} _\tau$ is always surjective and does not depend on $u$, $M$, or $J$. Thus, by the five lemma, surjectivity of $\bar{d}_u$ and $\n$ are equivalent conditions.

\begin{definition}

Let $E$ be a holomorphic line bundle. Let $\nabla^{0,1}: \Gamma(E) \m \Omega^{0,1}(E)$ be the anti-holomorphic part of $\nabla$, a Hermitian connection. We will call a first order differential operator $L: \Gamma(E) \m  \Omega^{0,1}(E)$ a generalized $\bar d$-operator if $L=\nabla^{0,1} + a$ with $a \in \Omega^{0,1}(End_{\R} E)$. 

\end{definition}

\begin{theorem}

If $E$ is a holomorphic line bundle on $\s$, $L: \Gamma(E) \m  \Omega^{0,1}(E)$ a generalized $\bar d$ operator, and $c_1(E) \geq -1$, then $L$ is surjective. Additionally assume that $c_1(E) \geq k-1$. For any collection of k points $x_1$, . . . $x_k$, let $V$ be the subspace of $\Gamma(E)$ of sections vanishing at each $x_i$. Then $L$ restricted to $V$ is surjective. 

\end{theorem}

The proof of the above theorem is due to Hofer, Lizan, and Sikorav in [HLS]. The proof involves showing that for any generalized $\bar d$ operator $L$, there is another holomorphic structure on $E$ such that $L=\nabla^{0,1}$ for some Hermitian connection. This reduces the problem to algebraic geometry and then the proof follows from Serre duality and positivity of intersection. They proved that $\n$ is a generalized $\bar d$ operator and deduce the following corollary.

\begin{corollary}

If $c_1(A) \geq 1$, then $\M_A(\s,(M,J))$ is a smooth manifold in a neighborhood of all points $[u]$ with $u$ an immersion. 

\end{corollary}

\begin{proof}

Note that $c_1(A)=c_1 (T \s)+c_1(N) = 2+ c_1(N)$. Hence $c_1(N) \geq -1 $ so $\n$ is onto for all $u \in \M_A(\s,(M,J))$. 

\end{proof}

The above two theorems are what is referred to as ``automatic transversality.'' One can prove surjectivity of $\bar d$ operators without restricting to a ``generic'' subset of almost complex structures.

If we are instead interested in the moduli space of curves passing through a finite set $Y$, we use the following construction. Let $u: \s \m M$ be a $J$-holomorphic immersion. Let $X=\{x_1, . . . x_n \}$ be a finite subset of $\s$. We can also view $X$ as a divisor. For simplicity, assume $u|_X$ is injective. Let $Y=\{u(x_1), u(x_2), . . \}$. Consider the sheaf $\mathcal O(-X)$ of holomorphic sections of the trivial complex line bundle vanishing on $X$. The minus sign indicates that we require sections to vanish. A plus sign would indicate that we would allow rational sections with simple polls on $X$. The sheaf $\mathcal O(-X)$ is isomorphic to the sheaf of holomorphic sections of a holomorphic vector bundle. Tensoring with a vector bundle is an exact functor. Since $0 \m T\s \m u^*TM \m N \m 0$ is exact, $0 \m T\s \otimes \mathcal O(-X) \m u^*TM \otimes \mathcal O(-X) \m N \otimes \mathcal O(-X) \m 0$ is an exact sequence of sheaves. Just as ker $\n$ is a model for the tangent space of $\M_A(\s,(M,J))$, the subspace of ker $ \n$ of sections vanishing on $X$ is a model for the tangent space of $\M_A(\s,(M,J),Y)$. Also the subspace of ker $\bar d_u$ of sections vanishing on $X$ gives a model for the tangent space of the subspace of $Hol_A(\s,(M,J))$ consisting of maps $v$ with $v(x_i)=y_i$.

\begin{theorem}
Let $\mathcal{J}$ denote the space of smooth almost complex structures compatible with $\omega$.  The space $\mathcal{J}$ is contractible.

\end{theorem}

See [G] or [MS] for a proof. Since $\mathcal{J}$ is path connected, we can take a path $J_t$ connecting any two almost complex structures $J_0$ and $J_1$. Let $\mathfrak{M}=\bigcup_t \M_A(\s,(M,J_t),Y)$. There is a natural map from $\pi: \mathfrak{M} \m [0,1]$ given by $\pi((u,t))=t$. 

\begin{theorem}

If $dim(M)=4$, $|Y|< (c_1(TM),A)$, and every curve is an immersion, then  $\mathfrak{M}$ is a manifold with boundary $\M_A(\s,(M,J_0),Y) \cup \M_A(\s,(M,J_1),Y)$ and $\pi : \mathfrak{M} \m [0,1]$ is a submersion.

\end{theorem}

This was first proven by Gromov in [G].  See [S] for another exposition of the proof. In addition to a criterion for $\pi$ being a submersion, there is a criterion for $\pi$ being proper.

\begin{theorem}
Let $A \in H_2(M)$ be a homology class with the following property. Assume for all non-zero homology classes $B$ and $C$ with
$B+C=A$, either $\M_B(\s,(M,J),Y)=\phi$ or $\M_C(\s,(M,J),Y)=\phi$. Under these assumptions, $\pi : \mathfrak{M}  \m \mathcal{J}$ is proper. 

\end{theorem}

This is a corollary to Gromov's compactness theorem first appearing in [G]. Gromov in [G] also observed the following sufficient condition for a moduli space of $J$ holomorphic curves to be empty.

\begin{theorem}
If $\w(A)<0$, then $\M_A(\s,(M,J))=\phi$.
\end{theorem}

Now let $M=\p$ and use the standard  Fubini-Study symplectic form. Denote the homology class given by $n$ times the fundamental class of the standard $\s \subset \p$ by $n$.

\begin{theorem}
If $u\in Hol_1(\s,(\p,J))$, then $u$ is an embedding.

\end{theorem}

The proof involves first showing that the adjuction formula from algebraic geometry applies to $J$ holomorphic curves in symplectic 4 manifolds [Mc]. This that implies the degree genus formula, $g=(d-1)(d-2)/2$, holds for embedded non-multiply covered $J$-holomorphic curves in $\p$. Since no degree one maps are multiple covers, this implies that genus zero maps of degree 1 are embeddings. In particular, the above theorem implies that all degree $1$ maps are immersion.

Let $J_0$ denote the almost complex structure associated to the standard complex structure on $\p$. From now on, take $Y=\{y_0\}$ to be a set with one element.

\begin{theorem}
For any compatible almost complex structure $J$, $\M_1(\s,(\p,J),\{y_0\})$ is diffeomorphic to $\M_1(\s,(\p,J_0),\{y_0\})$.
\end{theorem}

\begin{proof} This appears in [G] and [S]. Take a path $J_t$ of compatible almost complex structures with $J_1=J$ and $J_0$ being the complex structure induced from the standard complex structure. If $1=b+c$, with $b$ and $c$ not $0$, then $\w(b)$ or $\w(c)$ is negative. Thus the map $\pi: \bigcup_t \M_1(\s,(\p,J_t),\{y_0\}) \m [0,1]$ is proper. Note that $dim_{\R}(\p)=4$, all degree one curves are immersions, and $(c_1(T\p),1)=3>|\{y_0\}|=1$. Thus we can apply Theorem 2.5. Hence the projection map $\pi: \bigcup_t \M_1(\s,(\p,J_t),\{y_0\}) \m [0,1]$ is also a submersion. Hence we can conclude, $\bigcup_t \M_1(\s,(\p,J_t),\{y_0\})$ is diffeomorphic to $[0,1] \times \M_1(\s,(\p,J_0),\{y_0\})$ and  $\M_1(\s,(\p,J_0),\{y_0\})$ is diffeomorphic to $\M_1(\s,(\p,J),\{y_0\})$.
\end{proof}

Note that the $PSL_2(\C)$ action on $Hol(\s,(M,J))$ does not restrict to an action on $Hol^*(\s,(M,J))$. However, the subgroup $Hol_1^*(\s,\s) \leq PSL_2(\C)$ does act on  $Hol^*(\s,(M,J))$ by precomposition. In order to view $Hol^*_1(\s,\s)$ as a group under composition, we need to use the base point condition $u(\infty)=\infty$.

\begin{theorem}
The map $Hol_1^*(\s,(\p,J)) \m Hol_1^*(\s,(\p,J))/Hol_1^*(\s,\s) = \M_1(\s,(\p,J),\{y_0\})$ is a principle $Hol_1^*(\s,\s)$ bundle.
\end{theorem}
\begin{proof}
It suffices to show that the projection map is a submersion and that the group action is free.  The action is free since degree one maps are embeddings and hence all maps are changed by non trivial reparameterizations. The projection map being a submersion is equivalent to the following condition involving $\bar d$ operators: $\bar d_{\tau}: \Gamma^{0,1}(T\s) \m \Omega^{0,1}(T\s)$ is surjective when restricted to the subspace $\Gamma^{0,1}(T\s)$ of sections vanishing at the point $\infty \in \s$. To see that this $\bar d$ operator condition is equivalent to the projection being a submersion, see [S]. The surjectivity of $\bar d_{\tau}$ restricted to the subspace of sections vanishing at $\infty$ follows from automatic transversailty (Theorem 2.2) since $c_1(T\s)=2>0$.

\end{proof}

\begin{theorem}
For any compatible almost complex structure $J$, $Hol^*_1(\s,(\p,J))$ is diffeomorphic to $Hol_1^*(\s,(\p,J_0))$.
\end{theorem}

\begin{proof}
Fix a path $J_t$ of compatible almost complex structures with $J_1=J$ and $J_0$ being the complex structure induced from the standard complex structure. Since degree one curves are embedded,  $Hol^*_1(\s,\s)$ acts freely on  $Hol^*_1(\s,(\p,J_t)$ for each $t$. Thus the space $Hol^*_1(\s,(\p,J_t)$ is the total space of a $Hol^*_1(\s,\s)$ principle bundle over $\M_1(\s,(\p,J_t),\{y_0\})$ for each $t$. Thus, we get a one parameter family of maps $f_t :\M_1(\s,(\p,J_t)) \m B Hol^*_1(\s,\s)$ and hence the bundle isomorphism type does not change as we vary $t$. Thus, $Hol^*_1(\s,(\p,J))$ is bundle isomorphic to and hence diffeomorphic to $Hol^*_1(\s,(\p,J_0))$. 
\end{proof}

Note that this diffeomorphism comes from an isotopy through the space of based continuous maps. The homotopy type of the space of degree one rational holomorphic maps is straightforward to compute for the standard complex structure on complex projective space. 

\begin{theorem}
The space $Hol_1^*(\s,({\mathbb{P}^n},J_0))$ deformation retracts onto a subspace homeomorphic to $S^{2n+1}$.
\end{theorem}

See [CCMM2] for a proof.

\begin{corollary}
For any compatible almost complex structure $J$, the space $Hol^*_1(\s,(\p,J))$ deformation retracts onto a subspace homeomorphic to $S^3$.
\end{corollary}

The spaces $Hol_1^*(\s,(\p,J))$ are the fibers of the evaluation at $\infty$ map $ev:Hol_1(\s,(\p,J)) \m \p$. The previous arguments show that the fibers are all diffeomorphic. In fact, the evaluation map  $ev:Hol_1^*(\s,(\p,J)) \m \p$ is a smooth fiber bundle. This theorem is the goal of the remainder of the section.

\begin{definition}
Let $ev:Hol_1(\s,(\p,J)) \m \p$ be the map defined by $ev(u)=u(\infty)$.
\end{definition}

\begin{theorem}
For the almost complex structure $J_0$ induced from the standard complex structure on $\p$, $ev:Hol_1(\s,(\p,J_0)) \m \p$ is fiber bundle. 
\end{theorem}

\begin{proof}
The complex manifold $\p$ with the standard complex structure has a transitive group of holomorphic automorphisms. Hence the evaluation map is a fibration. 
\end{proof}

\begin{definition}
Let $\M(J)=Hol_1(\s,(\p,J))/Hol_1^*(\s,\s)$. Here $Hol_1^*(\s,\s)$ has the base point condition $u(\infty)=\infty \in \s$.
\end{definition}

The space $\M(J)$ is the moduli space of degree one rational $J$-holomorphic curves in $\p$ with one marked point. The evaluation at $\infty$ map is $Hol_1^*(\s,\s)$ invariant so it descends to a map $ev:\M(J) \m \p$. Since the quotient $PSL_2(\C)/Hol_1^*(\s,\s)$ is compact, the map $\M(J) \m Hol_1(\s,(\p,J))/PSL_2(\C)$ is proper. By Gromov's compactness theorem, $Hol_1(\s,(\p,J))/PSL_2(\C)$ is compact and hence so is $\M(J)$. By an identical proof to theorem 2.10, we see that $Hol(\s,(\p,J))$ is a $Hol_1^*(\s,\s)$ principle bundle over $\M(J)$. At the point $y_0 \in \p$, the fiber of the map $ev: \M(J) \m \p$ is $\M_1(\s,(\p,J),{y_0})$, the space of degree one rational $J$-holomorphic curves passing through the point $y_0$.

\begin{lemma}
The map $ev: \M(J) \m \p$ is a fiber bundle.
\end{lemma}

\begin{proof}
Since the map is proper, it is sufficient to show that the map is a submersion. By the work of [S], this map being a submersion at a curve $u$ is an equivalent condition to the normal $\bar d$ operator $\n$ of the curve $u$ being surjective. For all degree one curves, the normal bundle has first Chern number equal to $1$. Hence, by theorem 2.2, the operator $\n$ is surjective.
\end{proof}

From now on, fix a path $J_t$ of almost complex structures with $J_0$ being the standard almost complex structure, and $J_1=J$, the almost complex structure that we are investigating. 

\begin{lemma}
The map $ev:\M(J) \m \p$ is bundle isomorphic to  $ev: \M(J_0) \m \p$.
\end{lemma}

\begin{proof}
The map $\bigcup_t \M(J_t) \m \p \times [0,1]$ is a proper submersion of manifolds with boundary. Hence $\bigcup_t \M(J_t)$ is fiberwise diffeomorphic to $[0,1] \times \M(J_0)$.
\end{proof}

\begin{theorem}
For any compatible almost complex structure $J$, the map $ev:Hol_1(\s,(\p,J)) \m \p$ is fiber bundle. 
\end{theorem}

\begin{proof}
The map  $\bigcup_t Hol_1(\s,(\p,J_t)) \m  \bigcup_t \M(J_t)$ is a principle $Hol_1^*(\s,\s)$ bundle by a proof identical to the proof of theorem 2.10. Since the unit interval is connected, $Hol_1(\s,(\p,J)) \m  \M(J)$ is bundle isomorphic to $Hol_1(\s,(\p,J_0)) \m  \M(J_0)$. By the previous lemma,  $\M(J) \m \p$ is bundle isomorphic to $\M(J_0) \m \p$. Hence the map $ev:Hol_1(\s,(\p,J)) \m \p$ is fiberwise difeomorphic to $ev:Hol_1(\s,(\p,J_0)) \m \p$. By theorem 2.14, $ev:Hol_1(\s,(\p,J_0)) \m \p$ is a fiber bundle and hence so is $ev:Hol_1(\s,(\p,J)) \m \p$.

\end{proof}

\section{Review of properties of a combinatorial gluing map}\label{conclusions}

By Theorem 2.9, the topology of the space of degree one $J$-holomorphic maps is independent of almost complex structure. We will leverage this fact to get information about the topology of higher degree holomorphic mapping spaces. In this section we recall a way of constructing higher degree holomorphic maps out of lower degree maps in the case of the standard almost complex structure on $\mathbb{P}^n$. With the standard complex structure $J_0$, $Hol^*(\s,(\mathbb{P}^n,J_0))$ has an action of the little 2-disks operad $C_2$ [BM]. This allows the construction of higher degree maps out of degree one maps.

\begin{definition}
The little 2-disks operad $C_2$ is the operad whose $k$'th space $C_2(k)$ is defined to be the subspace of the space of embeddings of $\bigsqcup_k D_1$ into $D_1$, the closed unit disk in $\R^2$, of maps which are a composition of dilations and translations. The composition law is given by composing embeddings. The symmetric group $\Sigma_k$ acts on $C_2(k)$ by permuting the ordering of the disks. 
\end{definition}

See [M] for the definitions of an operad, and the definition of an action of an operad. In [M], they consider a similar operad involving rectangles instead of disks. The little $2$-disks operad acts on two fold loops. An embedding $g$ acts on $k$ elements $f_i \in \Omega^2 X$ as follows. An embedding $g: \bigsqcup_k D_1 \m D_1$ induces a collapse map $g': D_1 \m \bigvee_k S^2$ which sends every point outside the $k$ disks to the base point. Let $f':  \bigvee_k S^2 \m X$ be the map defined to be the function $f_i$ on the $i$'th sphere.  The operad action is defined by $(g,f_1,$... $f_k) \m f' \circ g'$.

In this section, holomorphic will always mean holomorphic with respect to the standard almost complex structure. We will state all results for $\mathbb{P}^n$ even though we are only interested in $\p$ in other sections. The action of the little 2-disk operad on $\Omega^2 \mathbb{P}^n$ does not restrict to an action on $Hol^*(\s,\mathbb{P}^n)$. In an unpublished manuscript, however, F. R. Cohen described an action of the little $2$-disks operad on the space $Hol^*(\s,\mathbb{P}^n)$ [BM]. View $\s$ as $\C \cup \{\infty\}$ and use homogenous coordinates on $\mathbb{P}^n$. The base point convention we use is $\infty \m [1:1:1: $ ... $ :1]$. In these coordinates, a based degree $k$ holomorphic function is given by $f(z)=[f_0(z):f_1(z):$ ... $f_n(z)]$ where the functions $f_i$ are monic degree $k$ polynomials with no root in common to all $n+1$ polynomials. A monic polynomial is determined by a positive divisor corresponding to its roots. Hence the space of monic polynomials is homeomorphic to the symmetric product of $\C$: $$SP(\C)=\bigsqcup_{k\geq 0} SP_k(\C) =\bigsqcup_{k\geq 0} \C^k/\Sigma_k$$

For any space $X$, the space $SP(X)$ has an abelian monoid structure, namely the free abelian monoid on the set $X$. This gives multiplication maps $SP(X)^k \m SP(X)$. Fix a homeomorphism from the unit disk to the complex plane. This induces a homeomorphism between $SP(\C)$ and $SP(D_1)$. An embedding $g:\bigsqcup_k D_1 \m D_1$ induces a map $g':SP(D_1)^k \m SP(D_1)$ as follows. Let $g_i : D_1 \m D_1$ be the restriction of $g$ to the $i$'th disk. The map $g_i$ induces a map $g'_i:SP(D_1) \m SP(D_1)$ given by applying the function $g_i$ to each point in a divisor. Let $g'$ be the product of the $g_i'$'s followed by the abelian monoid map $SP(\C)^k \m SP(\C)$. Applying this procedure to $n+1$ polynomials simultaneously gives an action of $C_2$ on $Hol^*(\s,\mathbb{P}^n)$.

The $C_2$ action on  $Hol^*(\s,\mathbb{P}^n)$ and $\Omega^2 \mathbb{P}^n$ respects degrees in the following sense. Let $k_1,$ .. $k_p$ be integers with $k_1+$. . . +$k_p=k$. The map $C_2(k) \times (Hol^*(\s, \mathbb{P}^n))^k \m Hol^*(\s, \mathbb{P}^n)$ restricts to a map: $$C_2(k) \times Hol_{k_1}^*(\s, \mathbb{P}) \times \text{. . . } \times  Hol_{k_p}^*(\s, \mathbb{P}) \m Hol_{k}^*(\s, \mathbb{P}^n)$$ Likewise the action on $\Omega^2  \mathbb{P}^n$ restricts to a map: $$C_2(k) \times \Omega^2_{k_1} \mathbb{P} \times \text{. . .} \times \Omega^2_{k_p} \mathbb{P}  \m \Omega^2_{k} \mathbb{P}^n$$ 
The inclusion map $i:Hol^*(\s,\mathbb{P}^n) \m \Omega^2 \mathbb{P}^n$ does not commute with the $C_2$ actions described on $Hol^*(\s,\mathbb{P}^n)$ and $\Omega^2 \mathbb{P}^n$. However, in [BM], a weaker statement is proven.

\begin{theorem} Let $k_1,$ .. $k_p$ be integers with $k_1+$. . . +$k_p=k$. The following diagram homotopy commutes:

$$
\begin{array}{ccccccccl}
 C_2(p) \times Hol_{k_1}^*(\s,\mathbb{P}^n) \times \text{... }  Hol_{k_p}^*(\s,\mathbb{P}^n) & \m  & Hol_k^*(\s,\mathbb{P}^n)               \\
i \downarrow  &    &i \downarrow             \\

 C_2(p) \times \Omega^2_{k_1} \mathbb{P}^n  \times \text{... } \Omega^2_{k_p} \mathbb{P}^n &\m & \Omega^2_k \mathbb{P}^n

\end{array}$$

\end{theorem}

This goes by the terminology $i:Hol^*(\s,\mathbb{P}^n) \m \Omega^2 \mathbb{P}^n$ is a morphism over $C_2$ up to homotopy. This $C_2$ action allows one to compare $Hol^*_k(\s,\mathbb{P}^n)$ for different $k$.

\begin{definition}
 Fix $v \in Hol^*_1(\s,\mathbb{P}^n)$ and $\mu \in C_2(2)$. Let $s:Hol^*_k(\s,\mathbb{P}^n) \m Hol^*_{k+1}(\s,\mathbb{P}^n)$ be given by $u \m \mu(v,u)$.
\end{definition}

Note that the map $s$ depends on $v$ and $\mu$. However, because $ Hol^*_1(\s,\mathbb{P}^n)$ and $ C_2(2)$ are connected, different choices yield homotopic maps. The map $s$ is called Segal's stabilization map. In [Se], Segal prove the following theorem:

\begin{theorem}
The map $s:Hol^*_k(\s,\mathbb{P}^n) \m Hol^*_{k+1}(\s,\mathbb{P}^n)$ and the map $i:Hol^*_k(\s,\mathbb{P}^n) \m \Omega^2_k \mathbb{P}^n$ are injective in homology for all dimensions and are isomorphisms on homology for dimensions $j \leq k(2n-1)$.

\end{theorem}

The study of the homology of $Hol^*_k(\s,\mathbb{P}^n)$ was continued in [CCMM] and [CCMM2]. They identify the entire homology of $Hol^*_k(\s,\mathbb{P}^n)$ even in the range where $i$ is no longer surjective. They also study the homotopical properties of a gluing map induced by the $C_2$ action on $Hol^*(\s, \mathbb{P}^n)$. The symmetric group $\Sigma_k$ acts on $C_2(k)$ by permuting the ordering of the disks and acts on $(Hol_{1}^*(\s,\mathbb{P}^n))^k$ by permuting components. The operad action induces a map from $C_2(k) \times  (Hol_{1}^*(\s,\mathbb{P}^n))^k \m  Hol_{k}^*(\s,\mathbb{P}^n)$ which descends to a map $\gamma: C_2(k) \times_{\Sigma_k}  (Hol_{1}^*(\s,\mathbb{P}^n))^k \m  Hol_{k}^*(\s,\mathbb{P}^n)$. In [CCMM], the authors prove that $\gamma$ is stably split.

\begin{definition}
Let $f:X \m Y$ be a continuous map between topological spaces. Let $\Sigma^{\infty}X$ and $\Sigma^{\infty}Y$ be the respective suspension spectra of $X$ and $Y$. The map $f$ is said to be stably split if there is a map of spectra $g:\Sigma^{\infty}Y \m \Sigma^{\infty}X$ such that $f \circ g$ is homotopic to the identity on $\Sigma^{\infty}Y$.

\end{definition}

\begin{theorem} The map $\gamma: C_2(k) \times_{\Sigma_k}  (Hol_{1}^*(\s,\mathbb{P}^n))^k \m  Hol_{k}^*(\s,\mathbb{P}^n)$ is stably split.
\end{theorem}

Note that the generalized homology theories of a space $X$ only depend on the stable homotopy type of $X$. Hence $\gamma$ induces a surjection on any generalized homology theory since it is stably split. The previous two theorems imply the following slightly unmotivated corollary.

\begin{corollary}

The composition $i \circ s \circ \gamma: C_2(k) \times_{\Sigma_k}  (Hol_{1}^*(\s,\p))^k \m \Omega^2_{k+1} \p$ induces a surjection on homology groups in dimensions $\leq 3k$.

\end{corollary}

This will be important since we will be able to show that the image in homology of $i \circ s \circ \gamma$ will be contained in the image in homology of $i: Hol_{k+1}^*(\s,(\p,J))\m \Omega^2_{k+1}\p$ for any compatible almost complex structure $J$.

\section{Construction of an analytic gluing map}

The purpose of this section is to construct a gluing map for $J$-holomorphic mapping spaces. Due to transversality requirements, we will have to restrict our attention to dimension $4$. We will not be able to construct a $C_2$ operad action $Hol^*(\s,(\p,J))$ or even a map similar to $\gamma$ from the previous section. However, we will construct a gluing map similar to $s \circ \gamma$. In this section we construct the gluing map, and in the next section we prove that it is a homology surjection by comparing it to the map in corollary 3.4. 

The gluing construction is based on the work of D. McDuff and D. Salamon in [MS] appendix A. Although they only glue two different curves together, they note that this can be done for more than two curves. Although our construction of a gluing map is very similar to that of [MS] and [S], the application is different. In [MS] and [S], gluing is viewed primarily as a converse to Gromov compactness, an attempt to describe a neighborhood of the boundary of moduli space. In contrast, we fix a gluing parameter and view gluing as an operation. Fix a curve $w \in  Hol^*_1(\s,(\p,J))$. Consider the following configuration space of points on $\s$ and holomorphic maps. 

\begin{definition}
Let $S$ be the subspace of $ (\C)^k \times (Hol_1(\s,(\p,J)))^k$ of points and maps $(x_1 . . .x_k, u_1 . . . u_k)$ with the $x_i$'s distinct  and $ u_i(\infty)=w(x_i)$.

\end{definition}

The space $S$ will contain the domain of our analytic gluing map. The collection of degree one curves $w$, $u_1$,... $u_k$ can be thought of as a connected singular degree $k+1$ curve and the gluing map involves deforming it to a nearby smooth curve. This process of deformation will require an implicit function theorem argument and hence a transversality condition. This transversality condition is the conclusion of theorem 4.2. If the transversality condition is met, the techniques of appendix A of [MS] produce a gluing map $\gamma^F: K \m Hol_{k+1}^*(\s,\p)$ for any subset $K \subset S$ which is contained in a compact set. The compactness condition does not cause any problems since $S$ deformation retracts onto a subset with compact closure.

\begin{lemma}
The set $S$ is homeomorphic to $F_k(\C) \times Hol^*_1(\s,(\p,J))^k$ with $ F_k(\C)$ being the configuration space of $k$ distinct ordered points in $\C$.
\end{lemma}

\begin{proof}
Consider the following diagram:

$$
\begin{array}{ccccccccl}
  &    & (Hol_1(\s,(\p,J)))^k               \\
    &    &ev \downarrow             \\

F_k(\C) &\overset{w',c}{\longrightarrow} &  (\p )^k

\end{array}$$
Here the map $w'$ is the map sending $(x_1,$... $x_k)$ to the point $(w(x_1),$... $w(x_k))$. The map $c$ is the constantant map sending every collection of points to the point  $(w(\infty),$... $w(\infty))$ and $ev$ is the evaluation at $\infty$ map. The map $ev$ is the projection map of a fiber bundle by theorem 2.17. The maps $w'$ and $c$ are homotopic so their pullbacks are homeomorphic. The pullback of $w'$ is $S$ and the pullback of $c$ is  $F_k(\C) \times Hol^*_1(\s,(\p,J))$.

\end{proof}

Let $F_k^{\epsilon}(B_1(0)) \subset F_k(\C)$ be the subspace of points that are at least $\epsilon$ apart and are contained in the interior of the ball of radius $1$. Fix a small $\epsilon$ such that $F_k^{\epsilon}(B_1(0))  \hookrightarrow F_k(\C)$ is a homotopy equivalence. Fix a homotopy equivalence between $S$ and $Hol^*_1(\s,(\p,J))^k \times F_k(\C)$ and fix an injective homotopy equivalence $S^3 \hookrightarrow Hol_1^*(\s,(\p,J))$.

\begin{definition}

Let $K \subset S$ be the image of $F_k^{\epsilon}(B_1(0)) \times  (S^3)^k \subset F_k(\C) \times Hol^*_1(\s,(\p,J))$ under the homeomorphism from $F_k(\C) \times Hol^*_1(\s,(\p,J))^k$ to $S$.

\end{definition}

We will build a gluing map $\gamma^F:K \m Hol_{k+1}^*(\s,(\p,J))$. The superscript $F$ is in honor of Floer.  Since $K$ has compact closure, we will be able to choose a single gluing parameter small enough for all maps in $K$. The only other ingredient needed to construct a gluing map is surjectivity of a certain operator needed to apply an implicit function theorem argument. Following the philosophy of [S], we prove the surjectivity of $\bar d$ operators by using automatic transversality arguments (Theorem 2.2). 

If $E$ is a vector bundle over a space $X$, $x\in X$, and $\xi \in \Gamma(E)$, let $\xi(x)\in E_x$ denote the value of the section at the point $x$. Degree one maps are embeddings so they have associated normal bundles. Let $N$ be the normal bundle of $w$.

\begin{definition}
Fix $s=(x_1,$... $x_k, u_1,$... $u_k)\in S$. Let $N_i$ be the normal bundle associated to the map $u_i$. Let $V_s$ be the subspace of $$\Gamma(w^*TM) \times \Gamma(u_1^*TM) \times \text{... } \times \Gamma(u_k^*TM)$$ consisting of sections $\xi,\xi_1, . . . \xi_k$ with the following matching conditions: $ \xi(\infty)=0$ and $  \xi_i(\infty)=\xi(x_i)$ for each $i$. Let $W_s$ be the following vector space: $$\Omega^{0,1}(w^*TM) \times \Omega^{0,1}(u_1^*TM)\text{. . .} \times \Omega^{0,1}(u_k^*TM)$$  Let $L_s: V_s \m W_s$  be given by $L (\xi,\xi_1, . . . \xi_k)= (\bar d_{w} (\xi),\bar d_{u_1} (\xi_1), . . . \bar d_{u_k} (\xi_k))$. In other words, $L_s$ is the restriction of $\bar d_w \times \bar d_{u_1}$... $\times \bar d_{u_k}$ to $V_s$.
\end{definition}

\begin{theorem} The linear transformation $L_s:V_s \m W_s$ is surjective for every $s\in S$.

\end{theorem}

\begin{proof}

By section 4.1 of [S], we have the following commuting diagram of exact sequences:

$$
\begin{array}{lcccccccl}
0  &\m &\Gamma(T\s) &\m  & \Gamma(w^*T\p) & \m & \Gamma(N) & \m & 0                \\
\downarrow  & & \bar{d}_\tau \downarrow  &    & \bar{d_w} \downarrow &  & \n \downarrow &  & \downarrow              \\

0 & \m  &\Omega^{0,1}(T\s) &\m & \Omega^{0,1}(w^*T\p) & \m &\Omega^{0,1}(N) &\m & 0
\end{array}$$

We also have similar diagrams for the normal bundle $N_i$ of the curves $u_i$. Let $\eta \in \Omega^{0,1}(w^* T\p)$ and $\eta_i \in \Omega^{0,1}(u_i^* T\p)$ be arbitrary. We will show that $(\eta, \eta_1, . . . \eta_k)$ is in the image of $L_s$.

We will first check that $\bar d_w:\Gamma(w^*T\p)\m\Omega^{0,1}(w^*T\p)$ is surjective. Note that $c_1(T\s)=2 \geq -1$ and $c_1(N)=1\geq -1$. Hence Theorem 2.2 implies the surjectivity of $\bar d_{\tau}$ and $\n$. The Five Lemma implies that $\bar d_w$ is surjective. Likewise $\bar d_{u_i}:\Gamma(u_i^*T\p) \m \Omega^{0,1}(u_i^*T\p)$ is surjective for each $i$. Let $\xi \in \Gamma(w^*T\p)$ and  $\xi_i \in \Gamma(u_i^*T\p)$ be such that $\bar d_w \xi = \eta$ and $\bar d_{u_i} \xi_i =\eta_i$.

If $(\xi,\xi_1,$ ... $\xi_k)$ is in $V_s$ then $L_s(\xi,\xi_1,$ ... $\xi_k)=(\eta,\eta_1$... $\eta_k)$ and we have shown that $(\eta,\eta_1,$... $\eta_k)$ is in the image of $L_s$. However, for a collection of sections to be in $V_s$, they need to also satisfy matching conditions which $(\xi,\xi_1,$ ... $\xi_k)$ may or may not satisfy. We will modify $(\xi,\xi_1,$ ... $\xi_k)$ by adding holomorphic sections. This will not change the value of the anti-holomorphic derivatives but will establish the matching conditions. In this proof,  we will call a section holomorphic if it is in the kernel of the relevant $\bar d$ operator.

Let $(w^*T\p)_{\infty}$ denote the fiber of the vector bundle $w^*T\p$ over the point in $\s$ named $\infty$. Let $a^n$ be the image of $a$ in $N_{\infty}$, the fiber of the line bundle $N$ over the point $\infty$. Let $\mathcal O(-\infty)$ denote the sheaf of holomorphic sections of the trivial complex line bundle vanishing at $\infty \in \s$. The complex dimension of the space of holomorphic sections of $N$ is $c_1(N)+1=2$. On the other hand, the complex dimension of the space of holomorphic sections of $N$ vanishing at $\infty$ is $c_1(N \otimes \mathcal O(-\infty))+1=1$. Thus, we can find a holomorpic section of $N$ not vanishing at $\infty$. In other words, there exists $\sigma^n \in \Gamma(N)$ with $\n \sigma^n =0$ and $\sigma(\infty) \neq 0 \in N_{\infty}$. Since $N_{\infty}$ has complex dimension one, we can find a scalar $k \in \C$ such that $k \sigma^n(\infty)=a^n$.

Since coker $\bar d_{\tau}=0$, the Snake Lemma implies that there is a short exact sequence:

$$0 \m \ker \bar d_{\tau} \m \ker \bar d_{w} \m \ker \n \m 0$$
Thus we can find a holomorphic section $\sigma$ of $w^*T\p$ mapping to $k \sigma^n$ in $\Gamma(N)$. View $(T\s)_{\infty}$ as a subspaces of  $(w^*T\p)_{\infty}$. The quotient $w^*T\p_{\infty} / T\s_{\infty}$ is naturally isomorphic to $N_{\infty}$. Observe that the image of $\sigma(\infty)-a$  in $N_{\infty}$ is zero. Thus $\sigma(\infty)-a \in T\s_{\infty}$. Since the dimension of the space of holomorphic sections of $T\s$ is greater than the dimension of the space of holomorphic sections of $T\s  \otimes \mathcal O(-\infty)$, we can find $\sigma^t \in \ker \bar d_{\tau}$ such that $\sigma^t({\infty}) \neq 0$. Since $T\s_{\infty}$ is one complex dimensional, by scaling we can assume that $\sigma^t({\infty}) = a-\sigma(\infty)$. Let $\xi'=-\sigma^t-\sigma + \xi$. We have $\bar d_w \xi'=\eta$ and $\xi'(\infty)=0$.

The same argument can be used to show that we can find holomorphic sections $\sigma_i $ of $u_i^*T\p$ with $\sigma_i(\infty)+\xi_i(\infty)=\xi'(x_i)$. The equation $\sigma_i(\infty)+\xi_i(\infty)=\xi'(x_i)$ makes sense since $u_i(\infty)=w(x_i)$ and hence we can identify $(u_i^*T\p)_{\infty}$ with $(w^*T\p)_{x_i}$. Let $\xi'_i=\xi_i-\sigma_i$. The collection of sections $(\xi',\xi'_1,$ ... $\xi'_k)$ is in $V_s$ since it satisfies all of the matching conditions. Since we have only added holomorphic sections, we have not changed any of the anti-holomorphic derivatives. Hence $L_s(\xi',\xi'_1,$ ... $\xi'_k)=(\eta,\eta_1$... $\eta_k)$. We have now established that $L_s$ is surjective. 

\end{proof}

In order to construct a gluing map $\gamma^F: K \m Hol_{k+1}^*(\s,(\p,J))$, we first define an ``approximate'' gluing map $\gamma^A: K \m \Omega^2_{k+1} \p$ that is close in weighted Sobolev norms to a map to $Hol^*_{k+1}(\s,\p)$. Use the short hand notation $\vec{x}=(x_1,$... $x_k)$ and $\vec{u}=(u_1,$... $u_k)$. The map $\gamma^A(\vec x, \vec u)$ will be the function $w$ for $z$ outside of small balls around the points $x_i$. In even smaller balls around each $x_i$, the function $\gamma^A(\vec x, \vec u)$ will be a reparameterized version of the function $u_i$. In the annuli in between, we will use cut off functions and the exponential map to transition smoothly. The map $\gamma^A$ can be defined explicitly by formulas. 

Let $\delta$ and $R$ be fixed parameters with $\delta$ small and $R$ large compared with $1/\delta$. Also require that $2/(\delta R) \leq \epsilon$ and $\delta <1$. Let $\rho: \C \m [0,1]$ be a smooth function such that $\rho(z)=0$ if $|z| \leq 1$ and $\rho(z)=1$ if $|z| \geq 2$.  Using the compactness of the closure of $K$, we can globally bound the $L^{\infty}$ norm of the derivative of the $u_i$'s. Using this bound, McDuff and Salamon showed that we can take $R$ large enough and $\delta$ small enough to find functions $W_i:B_{3/(\delta R)}(0): \m T_{w(x_i)}\p$ and $U_i:{\s-B_{\delta R/3}(0)} \m T_{w(x_i)}\p$, satisfying the following equations:

$$\exp_{w(x_i)}(W_i(z))=w(z)$$
$$\exp_{w(x_i)}(U_i(z))=u_i(z)$$ Note that we only require the above equations to hold when $z$ is in the domain of $U_i$ or $W_i$ respectively. Using the above functions and constants, we can now define $\gamma^A$.

\begin{definition}
Let $\gamma^A: K \m \Omega^2_{k+1} \p$ be defined by the formula $\gamma^A(\vec x, \vec u)(z)=$
$$\begin{cases}
 w(z) & \text{if } |z-x_i| \geq 2/(\delta R) \text{ for all i} \\
 u(R^2 (z-x_i)) & \text{if } |z-x_i| \leq \delta/(2  R) \\
\exp_{w(x_i)}(p(\delta R (z-x_i))W_i(z)& \text{if } \delta/(2  R) \leq |z-x_i| \leq  2/(\delta R)
\\ +p(\delta/ (R(z-x_i))U_i(R^2(z-x_i))) 
\end{cases}
$$

\end{definition}

To summarize, $\gamma^A(\vec x, \vec u)(z)$=$w(z)$ for $z$ outside of disks around the $x_i$'s. However, for $z$ near $x_i$, $\gamma^A(\vec x, \vec u)$ is a reparameterized version of $u_i$. In annuli around the $x_i$, we use the cutoff function $\rho$ to smoothly transition between $w$ and reparameterized versions of the $u_i$'s.

For fixed $\vec x$ and $\vec u$, $\gamma^A(\vec x, \vec u)$ is not $J$-holomorphic in the annuli around the $x_i$. Thus, $\gamma^A$ is not a map into $Hol_{k+1}^*(\s,(\p,J))$.  However, by taking $\delta$ small and $R$ large, McDuff and Salamon show that $\bar d^{nl}(\gamma^A(\vec x, \vec u))$ can be made arbitrarily small in weighted $L^p$ norms, for some choices of $\rho$. Using the compactness of the closure of $K$, $\bar d^{nl}(\gamma^A(\vec x, \vec u))$ can be made small for all $(\vec x, \vec u) \in K$ at the same time. This allows one to use an implicit function theorem (Theorem 3.3.4 of [MS]) argument to ``correct'' $\gamma^A$ and build a map $\gamma^F:K \m Hol_{k+1}^*(\s,(\p,J))$.

\begin{theorem}
For some choices of the function $\rho$ and for $\delta$ sufficiently small and $R$ sufficiently large (depending on $\delta$), there is a section $\xi$ of $T C^{\infty}(\s,\p)$ with property that for each for each $(\vec x, \vec u) \in K$,  the map $z \m \exp_{\gamma^A(\vec x, \vec u)(z)}(\xi({\gamma^A(\vec x, \vec u)})(z))$ is $J$ holomorphic. 
\end{theorem}

\begin{proof}
In appendix A of [MS], it is proven that the transversality condition proved in theorem $4.2$ implies the existence of such a section $\xi$ for large $R$ and small $\delta$.
\end{proof}

We can now define the gluing map $\gamma^F$.

\begin{definition}
Let $\gamma^F: K \m Hol_{k+1}^*(\s,(\p,J))$ be defined by the formula $\gamma^F(\vec x, \vec u)(z)=\exp_{\gamma^A(\vec x, \vec u)(z)}(\xi({\gamma^A(\vec x, \vec u)})(z))$.

\end{definition}

By theorem $4.3$, the above formula does define a $J$ holomorphic map for each pair $(\vec x, \vec u) \in K$. To see that the degree of $\gamma^A(\vec x, \vec u)$ and $\gamma^F(\vec x, \vec u)$ is $k+1$, see section 5 or [MS] appendix A.

\section{Homotopy between gluing maps}\label{conclusions}

In section $3$ we recalled the construction of a gluing map $\gamma: C_2(k) \times_{\Sigma_k}  (Hol_{1}^*(\s,\mathbb{P}^n))^k \m  Hol_{k}^*(\s,\mathbb{P}^n)$. This gluing map involves the juxtaposition of roots of polynomials. By the work of F. R. Cohen, R. Cohen, Mann, and Milgram, the image in $H_*(\Omega^2_k \mathbb{P}^n)$ of $\gamma \circ i$ is completely understood. In section 4, a different gluing map $\gamma^F: K \m Hol_{k+1}^k(\s,(\p,J))$ was constructed via analytical techniques. Recall that the space $K$ is a configuration space of points on $\s$ and intersecting $J$-holomorphic curves. The goal of this section is to relate the gluing map $\gamma^F$ to the well understood maps from section $3$. In this section we will consider four gluing maps, $\gamma^R$, $\gamma^L$, $\gamma^A$, and $\gamma^F$.

\begin{definition}

Let $\gamma^R: C_2(k) \times_{\Sigma_k} (Hol_{1}^*(\s,(\p,J_0)))^k \m  \Omega^2_{k+1} \p$ be the map defined by $i \circ s \circ \gamma$.

\end{definition}

See section 3 for definitions of $i,s$ and $\gamma$. The notation $R$ indicates that this map involves juxtaposition of roots of polynomials. Fix $v \in \Omega^2_1 \p$ and $\mu \in C_2(2)$.

\begin{definition}
Let $l:\Omega^2_k \p \m \Omega^2_{k+1} \p$ be given by $l(u)=\mu(v,u)$. 
\end{definition}

The map $l$ is analogous to the stablization map $s$ from section 3 except that $l$ is a map between continuous mapping spaces instead of holomorphic mapping spaces. Like $s$, $l$ depends on $\mu$ and $v$ but different choices yield homotopic functions. However, unlike the map $s$, $l$ is a homotopy equivalence. 

\begin{definition}

Let $\gamma^L: C_2(k) \times_{\Sigma_k} (Hol_{1}^*(\s,(\p,J_0)))^k \m  \Omega^2_{k+1} \p$ be the gluing map induced by $C_2$ operad action on $\Omega^2 \p$ post composed with $l$.

\end{definition}

The map $\gamma^L$ is similar to $\gamma^R$. However, the operad action used in $\gamma^L$ is the standard action on twofold loop spaces as opposed to the action on holomorphic mapping spaces involving roots of polynomials. 

Recall, in section 4 we constructed a set $K$ of points in $\C$, $\vec x =(x_1$... $x_k)$, and degree one holomorphic maps, $\vec u=(u_1,$... $u_k)$, with prescribed intersections with a fixed degree one map $w$. The space $K$ is homeomorphic to $F^\epsilon_k(B_1(0)) \times (S^3)^k$. Note that there are inclusions $F^\epsilon_k(B_1(0)) \hookrightarrow F_k(\C) \hookrightarrow C_2(k)$ and $S^3 \hookrightarrow Hol_1^*(\s,(\p,J))$ which induce homotopy equivalences. In section 4, we also constructed a map $\gamma^A: K \m  \Omega^2_{k+1} \p$ called the approximate gluing map and a map $\gamma^F: K \m {Hol_{k+1}(\s,(\p,J))}$. Observe that the symmetric group $\Sigma_k$ acts  diagonally on $K$ by permuting the indices of the points $x_i$'s and functions $u_i$'s.

\begin{lemma}
The maps $\gamma^A: K \m \Omega^2 \p$ and $\gamma^F : K \m Hol_{k+1}(\s,(\p,J))$ are $\Sigma_k$ invariant maps.
\end{lemma}

\begin{proof}
The action of $\Sigma_k$ does not change what function is glued in at what point. Instead $\Sigma_k$ relabels the configuration and the holomorphic functions. Examining the explicit formula for $\gamma^A$ in the previous section, we see that $\gamma^A$ is $\Sigma_k$ invariant. The map $\gamma^F$ is determined by the values of $\gamma^A$ so it is also $\Sigma_k$ invariant.
\end{proof}

Thus, we can view $\gamma^A$  as a map from $K/\Sigma_k$ to $\Omega^2_{k+1} \p$ and $\gamma^F$ as a map from $K/\Sigma_k$ to $Hol_{k+1}^*(\s,(\p,J))$. We will abuse notation and mean the induced maps on $K/\Sigma_k$ whenever we write $\gamma^A$ or $\gamma^F$ from now on.

\begin{theorem} The maps $\gamma^A$ and $i \circ \gamma^F$ are homotopic as maps from $K/{\Sigma_k}$ to $\Omega^2_{k+1} \p$.
\end{theorem}

\begin{proof}

Let $\gamma_t:K/{\Sigma_k \m \Omega^2\p}$ be the map given by the formula $\gamma_t(\vec x, \vec u)(z)=\exp_{\gamma^A(\vec x, \vec u)(z)}(t\xi({\gamma^A(\vec x, \vec u)(z)})(z))$ for $(\vec x, \vec u) \in K/\Sigma_k$. Here $\xi$ is the family of vector fields introduced in Theorem 4.3. When $t=1$, $\gamma_1=\gamma^F$ since $\gamma^F$ was defined by the formula $\gamma^F(\vec x, \vec u)=\exp_{\gamma^A(\vec x, \vec u)(z)}(\xi({\gamma^A(\vec x, \vec u))(z)}(z))$. Since for any point $x$ in a manifold, $\exp_x(0)=x$, we have $\gamma_0=\gamma^A$. Thus $\gamma_t$ gives a homotopy between $i \circ \gamma^F$ and $\gamma^A$.

\end{proof}

\begin{theorem} There is a homotopy equivalence $g:K/\Sigma_k \m C_2(k) \times_{\Sigma_k} (Hol_{1}^*(\s,(\p,J_0)))^k $ such that $\gamma^A : K/\Sigma_k \m \Omega^2_{k+1} \p$ is homotopic to $\gamma^L \circ g:K / \Sigma_k \m \Omega^2_{k+1} \p$.

\end{theorem}

\begin{proof} 

We will construct this homotopy in two steps. In step 1, we will homotope $\gamma^A$ to a function $\gamma^2$ which does not involve the bump function $\rho$. In step 2, we will extend $\gamma^2$ to a function defined on a continous mapping space, construct the homotopy equivalence $g:K/\Sigma_k \m C_2(k) \times_{\Sigma_k} (Hol_{1}^*(\s,(\p,J_0)))^k$ and prove that $\gamma^2$ is homotopic to $\gamma^L \circ g$.
$$ $$

Step 1:

The purpose of this step is to homotope $\gamma^A:K/\Sigma_k \m \Omega^2 \p$ to a map  $\gamma^2:K/\Sigma_k \m \Omega^2 \p$ that does not involve the bump function $\rho$ or the exponential map. For fixed $(\vec x, \vec u) \in K$, the function $\gamma^A(\vec x, \vec u)$ is a reparameterized version of $u_i$ in a disk of radius $\delta/(2R)$ around each $x_i$, $w$ outside of balls of radius $2/(\delta R)$ around the $x_i$'s, and a function defined using the bump function $\rho$ to smoothly transition in the annuli around $x_i$ of radii $\delta/(2R)$ and $2/(\delta R)$. We will first homotope $\gamma^A$ to a map $\gamma^1$ which is constant taking the value $u_i(\infty)=w(x_i)$ in the annuls of radii $\delta/(2R)$ and $2/(\delta R)$ around $x_i$. This function will not involve the bump function $\rho$. Then we will shrink the annuli to circles yielding the function $\gamma^2$ which is a reparameterized version of $w$ outside balls of radius $1/R$ around the $x_i$, and reparameterized versions of $u_i$ inside the balls of radius $x_i$ and constant on the circles of radius $1/R$ around the points $x_i$. We will describe the homotopy between $\gamma^A$ and $\gamma^1$ with explicit formulas. 

There are three fixed parameters that we used in section 4 to build $\gamma^A$: $\epsilon$, $R$, and $\delta$. The constant $\epsilon$ is the minimum distance allowed between the $x_i's$. Hence we required that $2/(\delta R) \leq \epsilon$. We also fixed a function $\rho: \C \m [0,1]$ which is $0$ for $|z| \leq 1$ and $1$ for $|z| \geq 2$. In section 4, we defined $\gamma^A$ by the formula:
$$\gamma^A(\vec x, \vec u)(z)=$$
$$\begin{cases}
 w(z) & \text{if } |z-x_i| \geq 2/(\delta R) \text{ for all i} \\
 u(R^2 (z-x_i)) & \text{if } |z-x_i| \leq \delta/(2  R) \\
\exp_{w(x_i)}(p(\delta R (z-x_i))W_i(z)& \text{if } \delta/(2  R) \leq |z-x_i| \leq  2/(\delta R)
\\ +p(\delta/ (R(z-x_i))U_i(R^2(z-x_i))) 
\end{cases}
$$
Recall $W_i$ and $U_i$ were defined in section 4 so that $\exp_{w(x_i)}(W_i(z))=w(z)$ and $\exp_{w(x_i)}(U_i(z))=u_i(z)$.

The map $\gamma^2$ which is the goal of this step is defined as follows. Fix a diffeomorphism $f: B_{1/R}(0) \m \C$. Fix a family of diffeomorphism $f_{\vec x}: \C - \bigcup B_{1/R}(x_i) \m \C - \bigcup \{x_i\}$ that depends continuously on $\vec x = (x_1, . . . x_k) \in F_{\epsilon}(\C)$. Define $\gamma^2$ by the formula:
$$\gamma^2(\vec x, \vec u)(z)=$$
$$\begin{cases}
u_i(f(z-x_i)) & \text{ if } |z-x_i| \leq 1/R \\
w(f_{\vec x}(z)) & \text{otherwise}
\end{cases}
$$

In order to homotope $\gamma^A$ to the function $\gamma^1$, described above, we need to make it constant in the annuli around the points $x_i$. This will involve composing the functions $u_i$'s with a function so that the composition sends an entire neighborhood of $\infty$ to the base point. Likewise, we compose $w$ with a function such that the composition is constant in neighborhoods of the $x_i$'s. More explicitly, we will use the following formula for our homotopy. Let $h^t: B_{1/R}(0) \m \s$ be a family of maps depending continuously on $t\in[0,1]$ with the following properties:  $h^0(z)=R^2z$, $|h^t(z)| \geq R \delta /2$ for $z$ with $|z| \geq \delta/(2R)$, and $h^1(z)=\infty$ for $z$ with $|z| \geq \delta/(2R)$. Also fix a family of maps $h^t_{\vec x}: \C - \bigcup B_{1/R}(x_i) \m \s$ that depends continuously on $\vec x \in F_{\epsilon}(\C)$ and $t\in[0,1]$ with the following properties: $h^0_{\vec x}(z)=z$, $h^t_{\vec x}(z)$ is in the annulus $1/R \leq |z-x_i| \leq 2/(\delta R)$ if $z$ is also in that annulus, and $h^1_{\vec x}(z)=x_i$ if $z$ is in the annulus $1/R \leq |z-x_i| \leq 2/(\delta R)$. Let $\gamma^t$ be defined by the formula:

$$\gamma^t(\vec x, \vec u)(z)=$$
$$\begin{cases}
 w(h^t_{\vec x}(z)) & \text{if } |z-x_i| \geq 2/(\delta R) \text{ for all i} \\
 u(h^t(z-x_i)) & \text{if } |z-x_i| \leq \delta/(2  R) \\
\exp_{w(x_i)}(p(\delta R (z-x_i))W_i(h^t_{\vec x}(z)) & \text{if } \delta/(2  R) \leq |z-x_i| \leq  2/(\delta R)
\\ +p(\delta/ (R(z-x_i))U_i(h^t(z-x_i))) 
\end{cases}
$$

Note $\gamma^0=\gamma^A$. The map $\gamma^1$ matches our description of $\gamma^1$ from above since it is constant in the annuli $\delta/(2  R) \leq |z-x_i| \leq  2/(\delta R)$. To see this, note that for $z$ in $\delta/(2  R) \leq |z-x_i| \leq  2/(\delta R)$, $h^1(z-x_i)=h^1_{\vec x}(z)=w(x_i)=u_i(\infty)$. Since $U_i(u_i(\infty))=0\in T_{u_i(\infty)}\p$ and $W_i(w(x_i))=0 \in T_{w(x_i)} \p$, for $z$ with $\delta/(2  R) \leq |z-x_i| \leq  2/(\delta R)$, we have: $$\gamma^1(\vec x, \vec u)(z)=$$ $$\exp_{w(x_i)}(p(\delta R (z-x_i))W_i(h^t_{\vec x}(z)+p(\delta/ (R(z-x_i))U_i(h^t(z-x_i)))=$$
$$\exp_{w(x_i)}(p(\delta R (z-x_i))0+p(\delta/ (R(z-x_i))0)=\exp_{w(x_i)}(0)=w(x_i)$$
Hence the map $\gamma^1(\vec x, \vec u)$ is a reparameterized version of $u_i$ inside each annulus, and a reparameterized version of $w$ outside the annuli and constant in the annuli. By shrinking the annuli down to circles, we can homotope $\gamma^1$ to the map $\gamma^2$ described above. Thus $\gamma^A$ is homotopic to $\gamma^2$.

$$ $$

Step 2:

The map $\gamma^2$ can be extended to a space of continuous maps satisfying the intersection conditions that functions in $K$ satisfy. To define this continuous mapping space, consider the following diagram:

$$
\begin{array}{ccccccccl}
  &    & SP_k(Map_1(\s,\p)               \\
    &    &ev \downarrow             \\

F_k^{\epsilon}(B_1(0))/\Sigma_k   &\overset{v'}{\longrightarrow} &  SP_k(\p)  

\end{array}$$

The map $v'$ is defined as follows. For $v:\s \m \p$ with $v(\infty)=[1:1:1]=w(\infty)$, let $v':F_k^{\epsilon}(B_1(0))/\Sigma_k \m SP_k(\p) $ be the restriction of the induced map $SP_k(\s) \m SP_k(\p)$. Let $ev$ be the evaluation map at $\infty \in \s$. Let $P(v)$ denote the pullback of $SP_k(\s,\p) \m SP_k(\p)$ under the map $v'$. There is a map $F_v:P(v) \m \Omega_{k+\deg(v)}^2 \p$ defined as follows. For $(\vec x, \vec u) \in P(v)$, let $F_v(\vec x, \vec u)=$

$$\begin{cases}
u_i(f(z-x_i)) & \text{ if } |z-x_i| \leq 1/R \\
v(f_{\vec x}(z)) & \text{otherwise.}
\end{cases}
$$
Here the maps $f$ and $f_{\vec x}$ are the same as those used in the definition of $\gamma^2$ from step 1. Let $P^J(v)$ be space defined in the same way as $P(v)$ except replacing $Map_1(\s,\p)$ with $Hol_1(\s,(\p,J))$. There is an inclusion map $i:P^J(v) \m P(v)$. The space $P^J(w)$ is the same as the space $S/{\Sigma_g}$ where $S$ is defined in section 4. Recall that the inclusion $K/\Sigma^k \hookrightarrow S/\Sigma_g$ is a homotopy equivalence. The inclusion map $K \hookrightarrow P^J(w)  \hookrightarrow P(w)$ followed by $F_v: P(w) \m  \Omega_{k+1}^2 \p$ is the map $\gamma^2:K  \m  \Omega_{k+1}^2 \p$ defined in step 1.

Now we will relate the function $\gamma^L$ to this construction involving pullbacks. Let $\hat w$ be a degree one continuous map with  $\hat{w}(z)=[1:1:1]$ for $z=\infty$ and $z \in B_1(0)$. The space $P(\hat w)$ is the space $(\Omega_1^2 \p)^k \times_{\Sigma_g} F_k^{\epsilon}(B_1(0))$ and $P^{J_0}(\hat w)$ is the space $F_k^{\epsilon}(B_1(0)) \times_{\Sigma_k} Hol_{1}^*(\s,(\p,J_0))^k$ since $\hat w'$ is constant. Let $g':F_k^{\epsilon}(B_1(0)) \m C_2(k)$ be the map that sends a collection of points to the collections of disks with centers around those points of radius $1/R$. The map $g'$ is a homotopy equivalence. The following diagram homotopy commutes: 
$$
\begin{array}{ccccccccl}
P^{J_0}(\hat w) & \overset{g' \times id}{\longrightarrow}  &  C_2(k) \times_{\Sigma_k} (Hol_{1}^*(\s,(\p,J_0)))^k              \\
i \downarrow  &    &\gamma^L \downarrow             \\
P(\hat w) &\overset{F_{\hat w}}{\longrightarrow} & \Omega^2_{k+1} \p
\end{array}$$
To prove this, note that for some choices of loop sum map and maps $f_{\vec x}$ and $f$, the above diagram commute on the nose.

From now on, assume $v$ is degree 1. Since $ev: Map_1(\s,\p) \m \p$ is a fibration, the homotopy type of $P(v)$ is independent of $v$. Since  $ev: Hol_1(\s,(\p,J)) \m \p$ is a fiber bundle, the homeomorphism type of $P^J(v)$ is independent of $v$. Additionally if $g'':P(w) \m P(\hat w)$ is a homotopy equivalence induced from a path between $w$ and $\hat w$, the following diagram homotopy commutes:
$$
\begin{array}{ccccccccl}
P(\hat w) & \overset{F_{\hat w}}{\longrightarrow}  & \Omega^2_{k+1}  \p           \\
g'' \downarrow  &    &id \downarrow             \\
P( w) &\overset{F_{ w}}{\longrightarrow} & \Omega^2_{k+1} \p
\end{array}$$
By the arguments used in section 2 to prove that the bundle isomorphism type of $ev:Hol_1(\s,(\p,J)) \m \p$ is independent of $J$, we see that there is a homeomorphism $g''': P^{J}(w) \m P^{J_0}(\hat w)$ making the following diagram homotopy commute:
$$
\begin{array}{ccccccccl}
P^{J}(w) & \overset{F_{w}}{\longrightarrow}  & \Omega^2_{k+1}  \p           \\
g''' \downarrow  &    &id \downarrow             \\
P^{J_0}( \hat w) &\overset{F_{ \hat w}}{\longrightarrow} & \Omega^2_{k+1} \p
\end{array}$$

Let $g:K/\Sigma \m C_2(k) \times_{\Sigma_k} Hol_1^*(\s,(\p,J_0))^k$ be defined by the compositions of the following maps: $$K/\Sigma \hookrightarrow S/\Sigma_k=P^J(w)$$ $$g''': P^J(w) \m P^{J_0}(\hat w)= F^{\epsilon}_k(B_1(0)) \times_{\Sigma_k} Hol_1^*(\s,(\p,J_0))^k$$ $$g' \times id:  F^{\epsilon}_k(B_1(0)) \times_{\Sigma_k} Hol_1^*(\s,(\p,J_0))^k \m  C_2(k) \times_{\Sigma_k} Hol_1^*(\s,(\p,J_0))^k$$ Assembling the information contained in the above homtotopy commuting diagrams yields the fact that $\gamma^2 : K/\Sigma_k \m \Omega^2_{k+1} \p$ is homotopic to $\gamma^L \circ g:K / \Sigma_k \m \Omega^2_{k+1} \p$. By step 1, $\gamma^A$ is also homotopic to $\gamma^L \circ g$.

\end{proof}

\begin{theorem} The maps $\gamma^R:C_2(k) \times_{\Sigma_k} (Hol_{1}^*(\s,(\p,J_0)))^k \m  \Omega^2_{k+1} \p$ and $\gamma^L:C_2(k) \times_{\Sigma_k} (Hol_{1}^*(\s,(\p,J_0)))^k \m  \Omega^2_{k+1} \p$ are homotopic.
\end{theorem}

\begin{proof}
This follows from Theorem 3.1.
\end{proof}

Having now shown that $i \circ \gamma^F$ is homotopic to $\gamma^R \circ g$, we will be able to conclude our central theorem.

\begin{theorem} For $J$ a compatible almost complex structure, $i:Hol_{k+1}^*(\s,(\p,J)) \m \Omega^2_{k+1} \p$ induces a surjection on homology groups for dimensions less than $3k$.
\end{theorem}

\begin{proof} 

Since $i \circ \gamma^F$ is homotopic to $\gamma^R \circ g$, the following diagram homotopy commutes:

$$
\begin{array}{ccccccccl}
K/\Sigma_k & \overset{\gamma^F}{\longrightarrow}  & Hol_{k+1}(\s,(\p,J))           \\
g \downarrow  &    &i \downarrow             \\
C_2(k) \times_{\Sigma_k} (Hol_{1}^*(\s,(\p,J_0)))^k& \overset{\gamma^R}{\longrightarrow} &\Omega^2_{k+1} \p
\end{array}$$
The map $g$ is a homotopy equivalence and $\gamma^R$ is surjective on $H_j$ for $j\leq 3k$. Thus the map $i:Hol_{k+1}^*(\s,(\p,J)) \m \Omega^2_{k+1} \p$  is surjective on $H_j$ for $j\leq 3k$ for any $J$ compatible with the standard symplectic form on $\p$. Re-indexing gives the version of this theorem stated in the introduction. 
\end{proof}

\paragraph{Remark} The fact that all four gluing maps are homotopic also has implications for the image in homology of $i:Hol_{k+1}^*(\s,(\p,J)) \m \Omega^2_{k+1} \p$ above the range $3k$. The map $\gamma:  C_2(k) \times_{\Sigma_k} (Hol_{1}^*(\s,(\p,J_0)))^k \m Hol_k^*(\s, (\p,J_0))$ defined in section 3 is a homology surjection in homological degrees. Hence the image in homology of $i:Hol_{k+1}^*(\s,(\p,J)) \m \Omega^2_{k+1} \p$ contains the image in homology of $l \circ i:Hol^*_k(\s,(\p,J_0)) \m \Omega^2_{k+1} \p$. Here $l:\Omega^2_k \p \m \Omega^2_{k+1} \p$ is the homotopy equivalence defined in Definition 5.2. Since $l$ and $i$ are injective on homology, and $Hol^*_{k}(\s,(\p,J_0)$ is a stable summand of $\Omega^2_{k+1}\p$ [CCMM], we see that the for any compatible almost complex structure $J$, the homology of $Hol_{k+1}^*(\s,(\p,J))$ surjects onto the homology of $Hol^*_k(\s,(\p,J_0))$ with $J_0$ the standard almost complex structure. See [CCMM] for more information about the homology of $Hol^*_k(\s,(\p,J_0))$.

\bibliographystyle{abbrv}
\bibliography{main}

\paragraph{References}

$$ $$

[A] Abreu M.: Topology of symplectomorphism groups of $S^2  \times S ^2$. Inv. Math. 131, 1-23 (1998).

[BM]  Boyer, C. P.,  Mann, B. M.: Monopoles, non-linear $\sigma$ models, and two-fold loop spaces. Comm. Math. Phys. 115, 571-594  (1988).

[CCMM] Cohen F.R., Cohen R.L., Mann, B.M., Milgram, R.J.: The topology of rational functions and divisors of surfaces. Acta Math. 166(3), 163-221 (1991).

[CCMM2] Cohen F.R., Cohen R.L., Mann, B.M., Milgram, R.J.: The homotopy type of rational functions. Mathematische Zeitschrift. 207, 213,37-47 (1993).

[G] Gromov M.: Pseudo-holomorphic curves in symplectic manifolds. Invent. Math. 82, 307-347 (1985).

[HLS] Hofer H., Lizan V., Sikorav J.-C.: On genericity for holomorphic curves in 4-dimensional almost-complex manifolds. J. Geom. Anal. 7, 149-159 (1998).

[HM] Harris J., Morrison  I.: Moduli of curves. Springer 1998.

[M] May J. P.: The Geometry of Iterated Loop Spaces. Lecture Notes in Mathematics, 271. Springer-Verlag, 1972.

[Mc] McDuff D.: The local behaviour of holomorphic curves in almost complex 4-manifolds. Journal of Differential Geometry 34, 311-358 (1991).

[MS] McDuff D., Salamon D.: J-holomorphic curves and quantum cohomology. Univ. Lect. Series 6, Amer. Math. Soc.,1994.

[Se] Segal G.B.: The topology of spaces of rational functions. Acta Math., 143, 39-72 (1979).

[S] Sikorav J.-C.: The gluing construction for normally generic J-holomorphic curves. ENS Lyon, March (2000).

\end{document}